\documentclass[a4paper,12pt,reqno]{amsart}
\usepackage[english]{babel}
\usepackage[T1]{fontenc}
\usepackage[left=1.1in,right=1.1in,top=1.2in,bottom=1.2in]{geometry}
\usepackage{times}
\usepackage{microtype}

\usepackage{commath,amssymb,amscd,mathrsfs,mathtools,dsfont,bbm,multirow,array,caption}
\usepackage[all]{xy}
\usepackage{enumerate}
\usepackage{longtable}

\usepackage[usenames,dvipsnames,svgnames,table]{xcolor}
\definecolor{red}{RGB}{255,25,25}
\definecolor{blue}{RGB}{25,50,200}
\usepackage{tikz-cd}
\usetikzlibrary{decorations.pathmorphing}

\usepackage[pagebackref,linktocpage]{hyperref}

\hypersetup{
pdftitle={\@title},			
pdfauthor={\authors},		
colorlinks=true,				
linkcolor=red,				
citecolor=MidnightBlue,		
filecolor=magenta,			
urlcolor=MidnightBlue			
}

\usepackage{cleveref} 

\newtheorem{theorem}{Theorem}[section]
\crefname{theorem}{Theorem}{Theorems}
\newtheorem{lemma}[theorem]{Lemma}
\crefname{lemma}{Lemma}{Lemmas}

\crefname{proposition}{Proposition}{Propositions}
\newtheorem{prop}[theorem]{Proposition}
\crefname{prop}{Proposition}{Propositions}
\newtheorem{corollary}[theorem]{Corollary}
\crefname{corollary}{Corollary}{Corollaries}

\crefname{cor}{Corollary}{Corollaries}

\crefname{conjecture}{Conjecture}{Conjectures}

\crefname{conj}{Conjecture}{Conjectures}
\newtheorem*{conj*}{Conjecture}
\crefname{conj}{Conjecture}{Conjectures}

\theoremstyle{definition}

\crefname{definition}{Definition}{Definitions}

\crefname{defn}{Definition}{Definitions}

\crefname{example}{Example}{Examples}

\crefname{notation}{Notation}{Notation}
\newtheorem*{notation*}{Notation}
\crefname{notation}{Notation}{Notation}

\crefname{problem}{Problem}{Problems}
\newtheorem{question}[theorem]{Question}
\crefname{question}{Question}{Questions}

\crefname{condition}{Condition}{Conditions}

\crefname{assumption}{Assumption}{Assumptions}

\theoremstyle{remark}
\newtheorem{rmk}[theorem]{Remark}
\crefname{rmk}{Remark}{Remarks}
\newtheorem*{rmk*}{Remark}
\crefname{rmk}{Remark}{Remarks}
\newtheorem{remark}[theorem]{Remark}
\crefname{remark}{Remark}{Remarks}

\crefname{fact}{Fact}{Facts}

\crefname{claim}{Claim}{Claims}
\newtheorem*{claim*}{Claim}
\crefname{claim}{Claim}{Claims}

\crefname{step}{Step}{Steps}

\crefname{case}{Case}{Cases}

\numberwithin{equation}{section}

\renewcommand{\lhd}{\trianglelefteq}

\newcommand{\doi}[1]{\href{https://doi.org/#1}{{\tt DOI:#1}}}

\newcommand{\bC}{\mathbf{C}}

\newcommand{\bF}{\mathbf{F}}

\newcommand{\bQ}{\mathbf{Q}}
\newcommand{\bR}{\mathbf{R}}
\newcommand{\bZ}{\mathbf{Z}}

\newcommand{\Aut}{\operatorname{Aut}}

\newcommand{\NS}{\operatorname{NS}}

\begin{document}

\title[Derived length of zero entropy groups]{Derived length of zero entropy groups acting on projective varieties in arbitrary characteristic--A remark to a paper of Dinh-Oguiso-Zhang}

\author{Sichen Li}
\address{School of Mathematical Sciences, Fudan University, Shanghai 200433, People's Republic of China
\endgraf Department of Mathematics, National University of Singapore, Block S17, 10 Lower Kent Ridge Road, Singapore 119076}
\email{\href{mailto:lisichen123@foxmail.com}{lisichen123@foxmail.com}}
\urladdr{\url{https://www.researchgate.net/profile/Sichen_Li4}}
\begin{abstract}
Let $X$ be a projective variety of dimension $n\ge1$ over an algebraically closed field of arbitrary characteristic.
We prove a Fujiki-Lieberman type theorem on the structure of  the automorphism group of $X$.
Let $G$ be a group of zero entropy automorphisms of $X$ and $G_0$ the set of elements in $G$ which are isotopic to the identity.
We show that after replacing $G$ by a suitable finite-index subgroup, $G/G_0$ is a unipotent group of the derived length at most $n-1$.
This result was first proved by Dinh, Oguiso and Zhang for compact K\"ahler manifolds.
\end{abstract}



\thanks{The author was partially supported  by  the China Scholar Council `High-level university graduate program'.}

\maketitle

\section{Introduction}
\label{sec:intro}
Let $X$ be a projective variety of dimension $n\ge1$ over an algebraically closed field $k$ of arbitrary characteristic.
It is well known that the automorphism group scheme $\Aut_X$ of a projective variety $X$ is locally of finite type over $k$ and $\Aut (X)=\Aut_X(k)$; in particular, the reduced neutral component $(\Aut_X^0)_{\mathrm{red}}$ of $\Aut_X$ is a smooth algebraic group over $k$ (cf. \cite[\S7]{Brion17}).
Denote $(\Aut_X^0)_{\mathrm{red}}(k)$ by $\Aut_0(X)$.
 If $G$ is a subgroup of $\Aut(X)$, define
\begin{equation*}
                                   G_0:=G\cap\Aut_0(X).
\end{equation*}
 Denote  by $\NS(X)=\mathrm{Pic}(X)/\mathrm{Pic}^0(X)$ the {\it N\'eron-Severi group} of $X$, i.e., the finitely generated abelian group of Cartier divisors on $X$ modulo algebraic equivalence.
 For a field $\bF=\bQ, \bR$ or $\bC$, the $\bF$-vector space $\NS_{\bF}(X)$ stands for $\NS(X)\otimes_{\bZ}\bF$; it is a finite-dimensional $\bF$-vector space.
Define the {\it first dynamcial degree} of an automorphism $g\in\Aut(X)$ as the {\it spectral radius}  of its natural action $g^*$ on $\NS_{\bR}(X)$, i.e.,
\begin{equation*}
                             d_1(g):=\rho(g^*|_{\NS_{\bR}(X)}):=\max\bigg\{|\lambda|: \lambda \text{ is an eigenvalue of } g^*|_{\NS_{\bR}(X)}\bigg\}.
\end{equation*}
We say that $g$ is of {\it positive entropy} if $d_1(g)>1$, otherwise it is of {\it zero~entropy}.
We call $G$ of {\it positive entropy}, if every element of $G\backslash\{\mathrm{id}\}$ is of positive entropy.

For a subgroup $G$ of the automorphism group $\Aut(X)$, we define the {\it zero-entropy subset of $G$} as
\begin{equation*}
N(G):= \bigg\{ g \in G : g \text{ is of zero entropy, i.e., } d_1(g) = 1 \bigg\}.
\end{equation*}
We call $G$ of  {\it zero entropy}, if $N(G) = G$.
For the study of dynamical degrees, we refer to \cite[\S4]{DS17} as a survey and \cite{Dang17,Truong18,Hu19} in arbitrary characteristic.

It is known that $\Aut(X)$ satisfies a Tits alternative and for solvable subgroups of $\Aut(X)$, the positive entropy part has a "bounded" size.
Precisely, the following result of the case of compact K$\mathrm{\ddot a}$hler manifolds or complex projective varieties with mild singularities was proved in \cite{Zhang09-Invent,CWZ14}, and of the case of projective varieties in arbitrary characteristic  was proved in \cite{Hu19}.
\begin{theorem}
\label{thm:Zhang-A}
Let $X$ be a projective variety of dimension $n\ge1$ and $G$ a subgroup of $\Aut(X)$.
Suppose that $G$ does not contain any non-abelian free subgroup.
Then there is a finite-index subgroup $G'$ of $G$ such that the quotient group $G'/N(G')$ is a free abelian of rank $r \le n-1-\max\{0,\kappa(X),\kappa(\omega_{X^v})\}$.
Here,  $\kappa(X)$ and $\kappa(\omega_{X^v})$ denote  the Kodaira dimension of $X$ and the Kodaira-Iitaka dimension of normalization $X^v$ of $X$.
\end{theorem}
For the definitions of $\kappa(X)$ and $\kappa(\omega_{X^v})$, we refer to \cite[Appendiexs A and B]{Patakfalvi18} or \cite[\S2.1.1]{Hu19}.
When $\mathrm{char}~k=0$, we know that $\kappa(X)=\kappa(\omega_{X^v})=\kappa(\widetilde{X})$, where  $\widetilde{X}\to X$ is a projective resolution.

If $X$ admits a group $G$ such that the rank of $G'/N(G')$ is maximal, i.e., equal to $n-1$, we will say that $X$ is a variety with maximal dynamical rank (MDR for short). Clearly, for such a variety,  we have $\kappa(X)\le n-1-\mathrm{rank}(G'/N(G'))=0$ by \cref{thm:Zhang-A}.
We refer to \cite{DS04,Zhang16} for more properties of these varieties over complex number field $\bC$.
The problem of classifying variety with MDR is still open when either char $k>0$ or $X$ is rational connected.
In \cite{HL19}, the authors tried to characterize the complex projective varieties of sub-maximal dynamical rank, i.e., equal to $n-2$.

In the note, we will focus our study on the group $N(G')$ in the last statement.
In order to simplify the notation, we consider groups $G$ such that every element of $G$ is of zero entropy.

Now recall that for a group $G$ and a non-negative integer $l$, the {\it l-th derived series} $G^{(l)}$ is defined inductively by
\begin{equation*}
                                           G^{(0)}:=G~~~\text{  and  }~~~ G^{(i+1)}:=[G^{(i)}, G^{(i)}].
\end{equation*}
By definition, $G^{(l)}=\{1\}$ for some non-negative integer $l$ exactly when $G$ is solvable.
 We call the minimum of such $l$ the {\it derived length} of $G$ (when $G$ is solvable) and denote it by
 \begin{equation*}
 \ell(G).
 \end{equation*}
A group $H$ is said to be {\it unipotent} if there is an injective homomorphism $\rho: H\to\mathrm{GL}(N,\bR)$  such that for every $h\in H$, the image $\rho(h)$ is upper triangular with all entries on the diagonal being 1.
Note that unipotent groups are solvable.
It is known that if a group $H$ is isomorphic to a subgroup of $\mathrm{GL}(N, \bR)$ whose elements have only eigenvalue 1, then $H$ is unipotent, see \cite[\S17.5]{Humphreys75}.
Below is our first result which slightly extends  \cite[Theorem 1.2]{DOZ18}.
 \begin{theorem}\label{MainThm}
Let $X$ be a projective variety of dimension $n\ge1$ and $G$ a subgroup of $\Aut(X)$ such that every element of $G$ is of zero entropy. Then
   \begin{enumerate}
    \item G admits a finite-index subgroup $G'$ such that, for any $1\le k\le n-1$, the natural map $G'/G_0'\longrightarrow G'|_{N_{\bR}^k(X)}$ is an isomorphism with image a unipotent subgroup of $\mathrm{GL}(N_{\bR}^k(X))$.
    \item For every finite-index subgroup $G'$ of $G$ such that $G'/G_0'$ is a unipotent group, the derived length of $G'/G_0'$ does not depend on the choice of $G'$  and is at most equal to $n-1$.
    \end{enumerate}
   \end{theorem}
 \begin{rmk}\label{rmk-proof}
(1) After replacing $H^{k,k}(X,\bF)$, $\mathcal{K}_i(X)$, K\"ahler cone,etc  as in \cite{DOZ18} by $N_{\bF}^k(X)$, $\mathrm{Nef}^i(X)$, ample cone,etc respectively, this proof in \cref{MainThm} uses the same argument in \cite{DOZ18}.
We will give a sketch of the proof in \cref{Proofs}.
 The main techniques used in this section are a Fujiki-Lieberman type  theorem (cf. \cref{MainThm2} below) and a higher-dimensional Hodge-index theorem for $\bR$-Cartier divisors (cf. \cite[Proposition 2.9]{Hu19}).

(2) In \cite{DOZ18}, Dinh, Oguiso and Zhang further established that
\begin{equation*}
\ell(G'/G_0')\le  n-\max\{\kappa(X),1\}
\end{equation*}
for a compact K$\mathrm{\ddot{a}}$hler manifold $X$.
It is hard to generalize this result in arbitrary characteristic.
Indeed, their argument essentially depends on the Deligne-Nakamura-Ueno theorem (cf. \cite[Corollary 2.4]{NZ09}), which is not known in positive characteristic, as far as we know.
\end{rmk}
Now let $X$ be a complex normal projective variety and $B$ a Cartier divisor on $X$.
Denote by $\Aut_{[B]}(X):=\{ g\in\Aut(X)| g^*[B]=[B]\}$.
When $X$ is smooth and $B$ is ample, Fujiki and Lieberman proved in \cite[Theorem 4.8]{Fujiki78} and \cite[Proposition 2.2]{Lieberman78} that $[\Aut_{[B]}(X):\Aut_0(X)]<\infty$.
Generally, let $G$ be a subgroup of $\Aut(X)$, such that for any $g\in G, g^*[B_g]=[B_g]$ for some big Cartier divisor $B_g$.
Dinh, Hu and Zhang proved in  \cite[Theorem 2.1]{DHZ15} that $G$ is virtually in $\Aut_0(X)$, i.e., $[G: G\cap \Aut_0(X)]<\infty$.
After replacing $g^*[B_g]=[B_g]$ by $g^*B_g\equiv_w B_g$ for some big Weil $\bR$-divisor $B_g$, Meng and Zhang showed in \cite[Theorem 1.2]{MZ18-Adv} that $G$ is virtually in $\Aut_0(X)$, where "$\equiv_w$" is the {\it weak numerical equivalence} (cf. \cite[Definition 2.2]{MZ18-Adv}).
Using a Hilbert scheme argument, Meng and Zhang proved in  \cite[Remark 2.6]{MZ18} that for an ample divisor $H$ on a projective variety $X$ in arbitrary characteristic, $[\Aut_{[H]}(X):\Aut_0(X)]<\infty$.

In the note, we  can further generalize \cite[Theorem 1.2]{MZ18-Adv} to the following.
\begin{theorem}
\label{MainThm2}
Let $X$ be a projective variety.
Let $G$ be a subgroup of $\Aut(X)$, such that for any $g\in G$, $g^*[B_g]=[B_g]$ for some big $\bR$-divisor $B_g$ in $N_{\bR}^1(X)$.
Then $G$ is virtually in $\Aut_0(X)$, i.e., $[G: G\cap \Aut_0(X)]<\infty$.
\end{theorem}
{\bf Notation and Terminology.}
We recall the definitions of $N_{\bR}^k(X)$ and $\mathrm{Nef}^k(X)$ in \cite{Fulton98}.
$N_{\bZ}^k(X)$ (or $N_k(X)_{\bZ}$) is the group of codimension (or dimension) $k$ algebraic cycles on $X$ modulo numerical equivalence.
We will use the vector spaces $N_{\bR}^k(X):=N_{\bZ}^k(X)\otimes_{\bZ}\bR, N_{\bC}^k(X):=N_{\bZ}^k(X)\otimes_{\bZ}\bC$ and $N_k(X)_{\bR}:=N_k(X)_{\bZ}\otimes_{\bZ}\bR$.
A $k$-cycle $Z$ is {\it effective}, if all of its defining coefficients are non-negative.
The corresponding numerical class $[Z]\in N_k(X)_{\bR}$ is called an {\it effective numerical class}.
We denote by $\overline{\mathrm{Eff}}_k(X)$ the closure of the convex cone generated by all effective numerical classes in $N_k(X)_{\bR}$.
It is called the  {\it pseudo-effective cone}  of $N_k(X)_{\bR}$.
The cone dual to $\overline{\mathrm{Eff}}_k(X)$ in $N_{\bR}^k(X)$ is called the {\it nef cone} $\mathrm{Nef}^k(X)$, which is a salient closed convex  cone of {\it full dimension} (i.e., it generates $N_{\bR}^k(X)$ as a vector space).
An element of $\mathrm{Nef}^k(X)$ is called a {\it nef class}.
In particular, $\mathrm{Nef}^1(X)$ is the usual nef cone $\mathrm{Nef}(X)$ consisting of all nef $\bR$-Cartier divisor classes.

Now let $X$ and $G$ be as in \cref{MainThm} and $H$  a finite-index solvable subgroup of $G$.
 We quote $\ell_{ess}(G,X)$ and $\ell_{\min}(H)$ in \cite[\S1]{DOZ18}.
The {\it essential length} of the action of $G$ on $X$ is defined by
 \begin{equation*}
                                  \ell_{ess}(G,X):=\ell(G'/G_0').
 \end{equation*}
Here, $G'$ is any finite-index subgroup of $G$ such that $G'/G_0'$ is a unipotent group.
This definition does not depend on the choice of $G'$, see \cite[Lemma 2.7]{DOZ18} or \cref{Lem1} below.
We also define
\begin{equation*}
                                   \ell_{\min}(H):=\min_{H'} \ell(H').
\end{equation*}
Here, $H'$ runs through all finite-index solvable subgroups of $H$, see also \cref{Lem1}.

When a group $G$ acts on a space $V$, we denote by $G|_V$ the image of the canonical homomorphism $G\to\Aut(V)$.
For instance, $\Aut(X)|_{N_{\bR}^k(X)}$ is the image of the canonical action of the automorphism group $\Aut(X)$ on $N_{\bR}^k(X)$.
For a normal subgroup $G_1\lhd G$, we set $(G/G_1)|_V=(G|_V)/(G_1|_V)$.
If $L$ and $M$ are two numerical classes, we denote by $L\cdot M$ or  $LM$.
We also identify $N_{\bR}^0(X)$ and $N_{\bR}^n(X)$ with $\bR$ in the canonical way.
So classes in these groups are identified to real numbers.
For a linear map $f: V\to V$, we denote by $||f||$ the {\it norm} of $f$.
\section{Proof of \cref{MainThm,MainThm2}}\label{Proofs}
We first quote two lemmas in \cite{DOZ18}, which will be used in the proof of \cref{MainThm}.
\begin{lemma}
\text{(cf. \cite[Lemma 2.7]{DOZ18})}
\label{Lem1}
Let $H$ be a unipotent group and let $H'$ be a finite-index subgroup of $H$. Then we have $\ell(H')=\ell(H)$.
In particular, we have $\ell_{\min}(H)=\ell(H)$.
\end{lemma}
\begin{lemma}\label{DOZ18Lem2.1}
\text{(cf. \cite[Lemma 2.1]{DOZ18})}
Let $V$ be a real vector space of finite dimension.
Let $\Gamma$ be a subgroup of $\mathrm{GL}(V)$.
Assume there is an integer $N\ge1$ such that $g^N$ is unipotent for every $g\in \Gamma$, i.e., their eigenvalues are 1.
Then there is a finite-index subgroup $\Gamma'$  of $\Gamma$ which is a unipotent subgroup of $\mathrm{GL}(V)$.
\end{lemma}
To prove \cref{MainThm2}, we needs a crucial lemma in \cite{MZ18-Adv}.
\begin{lemma}
\text{(cf. \cite[Proposition 2.9]{MZ18-Adv})}
\label{Meng-Zhang-Lem}
Let $f: V\to V$ be an invertible linear map of a positive dimensional real normed vector space $V$ such that $f^{\pm1}(C)=(C)$ for a closed convex cone $C\subseteq V$ which spans $V$ and contains no line.
Let $q$ be a positive number.
Then (1) and (2) below are equivalent.
\begin{enumerate}
\item[(1)] $f(x)=qx$ for some $x\in C^{\circ}$ ( the interior part of $C$).
\item[(2)] There exists a constant $N>0$, such that $\frac{||f^i||}{q^i}<N$ for any $i\in\bZ$.
\end{enumerate}
If (1) or (2) above is true, then $f$ is a diagonalizable linear map with all eigenvalues of modulus $q$.
\end{lemma}
\begin{proof}[\textup{\bf Proof of  \cref{MainThm2}.}]
Take an element $g\in G$.
Then $g^*[B_g]=[B_g]$ for some big $\bR$-divisor $B_g$ as an interior point in the pseudo-effective cone of $X$.
Note that the pseudo-effective cone and nef cone of $X$ are invariant by $\Aut(X)$.
By \cref{Meng-Zhang-Lem}, we get (2) in \cref{Meng-Zhang-Lem}.
Applying \cref{Meng-Zhang-Lem} to the nef cone ($=C$) of $X$, we obtain (1) in \cref{Meng-Zhang-Lem}.
So there is an ample $\bR$-divisor $H_g$ such that $g^*[H_g]=[H_g]$ and $g^*$ is a diagonalizable map with all eigenvalues of modulus 1.
By \cite[Lemma 3.5]{MZ18-Adv}, we may assume $H_g$ is an ample Cartier divisor.
Take $L:=N_{\bZ}^1(X)/\textup{(torsion)}$.
Let $D$ be a line bundle over $X$ and consider the {\it polarization map}
 \begin{equation*}
 \tau_D: \Aut(X)\to\mathrm{Pic}~(X), ~~g\longmapsto g^*(D)\otimes D^{-1},
 \end{equation*}
 which takes the identity to the trivial bundle, and hence $\Aut_0(X)$ to $\mathrm{Pic}^0(X)$.
This yields that $\Aut_0(X)|_L=\{\mathrm{id}\}$.
Let $r:=\mathrm{rank}(L)$.
Notice that the characteristic polynomial $f(x)$ of  $g|_L$ is a monic polynomial of degree $r$ over $\bZ$  whose all roots $\lambda$ are all of modulus 1.
 By Gauss's lemma, the minimal polynomial $p(x)$ of $\lambda$ has $p(x)|f(x)$ in $\bZ[x]$.
Then $p(x)$ is an irreducible monic polynomial over $\bZ$ whose all roots have absolute value 1 and $\deg(p(x))\le r$.
Thus,  all $\lambda$ are all roots of unity by Kronecker's Theorem.
So there is a minimal positive integer $d$ such that $\lambda^d=1$.
This yields that $p(x)$ is a cyclotomic polynomial of degree $\varphi(d)$ and $\varphi(d)\le r$.
Here $\varphi(d):=|\mathrm{Gal}(\bQ(\zeta_d)/\bQ)|$ is the Euler function.
There are finitely many $d$ with $\varphi(d)\le r$.
Let $m$ be their product, which is independent of $g$.
Then every eigenvalue of $g^m|_L$ is 1.
Moreover, $g^m|_L=\mathrm{id}$ since $g|_L$ is diagonalizable.
Take $\rho: G\to \mathrm{GL}(L\otimes_{\bZ} \bC)$, and $\ker{\rho}:=G_1$.
Now applying Burnside's theorem to $\mathrm{GL}(L\otimes_{\bZ} \bC)$,  $G|_L$ is finite and $[G:G_1]<\infty$.
In particular, $G_1|_L$ is trivial.
This yields that $G_1$ fixes any ample class $[H]$ and then $G_1\le\Aut_{[H]}(X)$.
Thus $G_1$ (and hence $G$) is virtually in $\Aut_0(X)$ by \cite[Remark 2.6]{MZ18}.
This completes the proof  of \cref{MainThm2}.
\end{proof}
The following lemma implies \cref{MainThm}(1).
\begin{lemma}\label{MainLem1}
\text{(cf. \cite[Lemma 2.8]{DOZ18})}
Let $X$ be a projective variety of dimension $n$. Let $G$ be a subgroup of $\Aut(X)$ with only zero entropy elements. Then there is a finite-index subgroup $G'$ of $G$ satisfying the following properties for every $1\le k\le n-1$.
\begin{enumerate}
\item The kernel of the canonical representation
\begin{equation*}
                                \rho_k: G'\longrightarrow \mathrm{GL}(N_{\bR}^k(X))
\end{equation*}
is equal to $G_0'$;
\item The image  of $\rho_k$ is a unipotent subgroup of $\mathrm{GL}(N_{\bR}^k(X))$.
\end{enumerate}
\end{lemma}
\begin{proof}
By  \cref{rmk-proof}(1) and  \cref{MainThm2}, the proof of \cref{MainLem1}(1) is the same as \cite[Lemma 2.8]{DOZ18}.
Here, we give another proof of \cref{MainLem1}(2), which is  slightly different from \cite[Proof of Lemma 2.8(2)]{DOZ18}.
We define {\it k-th  dynamical degree} (see eg. \cite[\S2.4]{Hu19}) by the natural pullback $g^*$ on $N_{\bR}^k(X)$ for any integer $0\le k\le n$.
Namely,
\begin{equation*}
                     d_k(g):=\rho\bigg( g^*|_{N_{\bR}^k(X)}\bigg)=\max\bigg\{ |\lambda| : \lambda \textup{ is an eigenvalue of } g^*|_{N_{\bR}^k(X)}\bigg\}.
\end{equation*}
Notice that $d_1(g)\le1$ since $g$ is of zero entropy.
By the log-concavity of dynamical degrees (cf. \cite[Corollary 2.11]{Hu19}), $d_k(g)\le d_1^k(g)$ for every $k$.
This yields that $d_k(g)\le1$ for every $k$.
Let $r_k:=\mathrm{rank} (N_{\bR}^k(X))$.
Notice that the characteristic polynomial $f_k(x)$ of  $g|_{N_{\bR}^k(X)}$ is a monic polynomial of degree $r_k$ over $\bZ$.
Hence, the product of all roots $\lambda_k$ of $f_k(x)$ is an integer.
So all $\lambda_k$ are algebraic integers of modulus 1 as $d_k(x)\le1$.
 By Gauss's lemma, the minimal polynomial $p_k(x)$ of $\lambda_k$ has $p_k(x)|f_k(x)$ in $\bZ[x]$.
 Then $p_k(x)$ is an irreducible monic polynomial over $\bZ$ whose all roots have absolute value 1 and $\deg(p_k(x))\le r_k$.
Thus,  all $\lambda_k$ are all roots of unity by Kronecker's Theorem.
So there is a minimal positive integer $d$ such that $\lambda_k^d=1$.
This yields that $p_k(x)$ is a cyclotomic polynomial of degree $\varphi(d)$ and $\varphi(d)\le r_k$.
Here $\varphi(d):=|\mathrm{Gal}(\bQ(\zeta_d)/\bQ)|$ is the Euler function.
There are finitely many $d$ with $\varphi(d)\le r_k$.
Let $m_k$ be their product, which is independent of $g$.
So $(g^{m_k})|_{N_{\bR}^k(X)}$ is unipotent.
According to \cref{DOZ18Lem2.1}, replacing $G$ by a finite-index subgroup, we have that $\rho_k(G)$ contains only unipotent elements of $\mathrm{GL}(N_{\bR}^k(X))$.
This completes the proof of \cref{MainLem1}(2).
\end{proof}
Let $G\le\Aut(X)$. We say that a rational map $f: X\dashrightarrow Y$ is {\it $G$-equivariant or $f$-equivariantly}, if there is a group homomorphism $\rho: G\to \Aut(Y)$ such that $f\circ g=\rho(g)\circ f$ for all $g\in G$.
In \cite[Lemma 2.9]{DOZ18}, the authors obtained a $G$-equivariant resolution by using a canonical resolution in \cite[Theorem 13.2]{BM97}.
However, the existence of the resolution of singularities is not known in positive characteristic to the best of our knowledge.
Note that the dynamical degrees in arbitrary characteristic is birational invariance (cf. \cite[Lemma 2.8]{Hu19}).
This motivates the following result, which  allows us to work with suitable birationally equivalent models.
\begin{prop}\label{birational}
\text{(cf. \cite[Lemmas 2.9 and 2.10]{DOZ18})}
Let $\pi: X_1\dashrightarrow X_2$ be a dominant rational map between projective varieties of the same dimension.
Let $G$ be a group acting on both $X_1$ and $X_2$ biregularly and $\pi$-equivariantly.
Then the following statements hold.
\begin{enumerate}
\item[(1)] Let  $\widetilde{X_1}$ be the normalization of the graph of the map  $\pi: X_1 \dashrightarrow X_2$.
 Then there is a birational map $p_1: \widetilde{X_1}\to X_1$ such that $\pi\circ p_1: \widetilde{X_1}\to X_2$ is a $G$-equivairant surjective morphism.
    In particular, $p_1^{-1}\circ G\circ p_1$  is a subgroup of $\Aut(\widetilde{X_1})$ and is isomorphic to $G$.
\item[(2)] Suppose $G|_{X_1}$ (or equivalently $G|_{X_2}$) is of zero entropy.
Then we have
\begin{equation*}
                                  \ell_{\mathrm{ess}}(G|_{X_1}, X_1)=\ell_{\mathrm{ess}}(G|_{X_2}, X_2).
\end{equation*}
Further, replacing $G$ by a finite-index subgroup, $(G|_{X_1})/(G|_{X_1})_0\cong (G|_{X_2})/(G|_{X_2})_0$.
\end{enumerate}
\end{prop}
\begin{remark}
The proof of \cref{birational}(1) is the same as \cite[Proof of Lemma 2.9]{DOZ18} without using the canonical resolution  in \cite[Theorem 13.2]{BM97}.
After replacing $H^{1,1}(X,\bR)$ by $N_{\bR}^1(X)$, the proof of \cref{birational}(2) is the same as  \cite[Lemma  2.10]{DOZ18} by using \cite[Lemma 2.8]{Hu19}.
\end{remark}
Chevalley's theorem on algebraic groups  asserts that every algebraic group over a perfect field is `built up' from a linear algebraic group and an abelian variety.
By Chevalley's theorem and \cite[Theorem 14.1]{Ueno75}, if a complex projective variety $X$ is not uniruled, then $\Aut_0(X)$ is an abelian variety.
This result is also known in positive characteristic as follows.
\begin{lemma}\label{non-uniruled}
\text{(cf. \cite[Proposition 7.1.4]{Brion17})}
Let $X$ be a projective normal variety which is not uniruled.
Then $\Aut_0(X)$ is an abelian variety.
\end{lemma}
The following \cref{derived length} implies \cref{MainThm}(2) by  \cref{Lem1,MainLem1}.
\begin{prop}\label{derived length}
Let $X$ and $G$ be as in \cref{MainThm}. Then we have
\begin{equation*}
                                          \ell_{\mathrm{ess}}(G, X)\le n-1.
\end{equation*}
\end{prop}
\begin{rmk}
By \cref{rmk-proof}(1), the proof of \cref{derived length} uses the same argument  as \cite[Proposition 3.1]{DOZ18}.
\end{rmk}
\begin{corollary}
\label{corollary}
\text{(cf.  \cite[Corollary 3.2 and Proposition 5.5]{DOZ18})}
Let $X$ be a projective normal variety of dimension $n\ge1$.
Let $G\le\Aut(X)$ be a group of zero entropy automorphisms.
Suppose that $\Aut_0(X)$ is commutative (for instance, this holds when $X$ is not uniruled, see \cref{non-uniruled}).
Then
   \begin{enumerate}
    \item  We have the following (in)equalities
    \begin{equation*}
                                    \ell_{\min}(G)\le\ell_{\min}(G/G_0)+1=\ell_{\min}(G|_{N_{\bR}^1(X)})+1=\ell_{\mathrm{ess}}(G, X)+1\le n.
    \end{equation*}
    \item  If $\ell_{\min}(G)=n$, then we have
    \begin{equation*}
                                             G_0\ne \{1\}~~~~\text{ and    }~~~ \ell_{\mathrm{ess}}(G, X)=n-1.
    \end{equation*}
    \end{enumerate}
\end{corollary}
\begin{remark}
After replacing $H^{1,1}(X,\bR)$ by $N_{\bR}^1(X)$, the proof of \cref{corollary} is the same as \cite[Corollary 3.2]{DOZ18}.
\end{remark}
\section{On the nilpotency class of zero entropy groups}
For a group $G$ and a non-negative integer $c$, the  {\it lower center series} $G_{(c)}$ is defined inductively by
\begin{equation*}
                                           G_{(0)}:=G~~~\textup{ and }~~~G_{(i+1)}:=[G_{(i)}, G].
\end{equation*}
By definition, $G_{(c)}=\{1\}$ for some non-negative integer $c$ exactly when $G$ is nilpotent. We call the minimum of such $c$ the {\it nilpotency class}  of $G$ (when $G$ is nilpotent) and denote it by $\mathfrak{c}(G)$.
 A normal series $G=G_1\rhd G_2\rhd\cdots\rhd G_n=1$ with each $G_i\lhd G$ and $G_i/G_{i+1}\le Z(G/G_{i+1})$ is a {\it central series}.
 It is well known that a group $G$ is nilpotent if and only if it has a central series $G=G_1\rhd G_2\rhd\cdots\rhd G_n=1$.
 \begin{lemma}\label{nilpotent}
\text{(cf. \cite[Lemma 2.1]{Oguiso06})}
Let $V$ be a positive dimensional vector space and $G$ a nilpotent (or solvable) subgroup of $\mathrm{GL}(V)$.
Then:
\begin{enumerate}
\item Any subgroup of $G$ and any quotient group of $G$ are nilpotent (or solvable).
\item The Zariski closure $\overline{G}$ of $G$  in $\mathrm{GL}(V)$  is also nilpotent (or solvable).
\end{enumerate}
\end{lemma}
\begin{proof}
When $G$ is solvable, this lemma is  \cite[Lemma 2.1]{Oguiso06}.
We only show (2) when $G$ is nilpotent.
It suffices to check that $\overline{[H,G]}=[\overline{H},\overline{G}]$ for $H\lhd G$.
Note that $[\overline{H},\overline{G}]$ is closed in $\overline{G}$ (see for instance \cite[\S17.2]{Humphreys75}).
Thus $\overline{[H,G]}\subset [\overline{H},\overline{G}]$.

Let us show the other inclusion.
Take $g\in G$.
Let us define the map $\theta_{g}$ by
\begin{equation*}
         \alpha_g: \overline{H} \to \overline{[H,G]}; g \mapsto f^{-1}g^{-1}fg.
\end{equation*}
Clearly, $\alpha_g$ is continuous and satisfies $\alpha_g(H)\subset [H,G]$.
Thus, $\alpha_g(\overline{H})\subset \overline{[H,G]}$.
Hence, $[\overline{H},G]\subset \overline{[H,G]}$.
Let $f\in\overline{H}$.
Let us define the map $\beta_f$ by
\begin{equation*}
                    \beta_f: \overline{G}\to \overline{G}; g \mapsto f^{-1}g^{-1}fg.
\end{equation*}
Clearly, $\beta_f$ is continuous and satisfies $\beta_f(G)\subset [\overline{H},G]$.
Since $[\overline{H},G]\subset \overline{[H,G]}$, we have $\beta_f(G)\subset \overline{[H,G]}$ as well.
Thus, $\beta_f(\overline{G})\subset\overline{[H,G]}$ and hence $[\overline{H},\overline{G}]\subset\overline{[H,G]}$.
\end{proof}
\begin{rmk}\label{rmk1}
In the proof of \cref{derived length}, the derived series of zero entropy groups follows from the abelian quotient groups $H_{i+1}/H_i$ as in \cite{DOZ18}.
Note that some lemmas in the proof of \cref{derived length} are also true after replacing the derived series by the lower center series.
For instance, $\mathfrak{c}(G'/G_0')$ does not depend on the choice of $G'$, where $G'$ is a finite-index subgroup of $G$ as in \cref{MainThm} such that $G'/G_0'$ is a unipotent group.
Using \cref{nilpotent}, this proof is similar as \cref{Lem1}.
Observe that it is hard to figure out a crucial result (which is similar to \cite[Proposition 2.6]{DOZ18}) for constructing a central series $G_i$ such that $G_i/G_{i+1}\le Z(G/G_{i+1})$.
These raise the following question.
\end{rmk}
\begin{question}\label{Que}
Let $X$ be a projective variety of dimension $n\ge1$ and $G$ a subgroup of $\Aut(X)$ such that all elements of $G$ have zero entropy.
Let $G'$ be a finite-index subgroup of $G$ such that $G'/G_0'$ is a unipotent group.
Is there a nonnegative integer $c$ such that $\mathfrak{c}(G'/G_0')$ is at most equal to $c$?
Moreover, is $c=n-1$ right?
\end{question}
For \cref{Que}, we quote the following result in \cite{Robinson96}, which is noticed by Fei Hu.
\begin{prop}
\text{(cf. \cite[5.1.12]{Robinson96})}
Let  $G$ be a  nilpotent group. Then $G$ is solvable, and we have:
\begin{equation*}
                                    \ell(G)\le \log_2 \mathfrak{c}(G)+1.
\end{equation*}
\end{prop}
\section*{Acknowledgments}
The author gratefully acknowledge De-Qi Zhang for many inspiring discussions and comments.
He would like to thank Fei Hu and Guolei Zhong for helpful communications and the anonymous referee for several suggestions.



\bibliographystyle{amsalpha}
\bibliography{../mybib}

\end{document}